\documentclass[12pt,reqno]{amsart}
\usepackage[utf8]{inputenc}
\usepackage{amsmath,amsthm,amssymb,url,mathdots,verbatim}

\usepackage{tikz}
\usepackage[centerboxes]{ytableau}
\usepackage[colorlinks=true]{hyperref}
\usepackage{MnSymbol}
\usepackage{eurosym}
\usepackage{graphicx}
\usepackage{scalerel}
\usepackage{enumitem}
\usepackage[normalem]{ulem}
\usepackage[capitalise]{cleveref}
\crefname{figure}{Figure}{Figures}

\numberwithin{equation}{section}

\newtheorem{theorem}{Theorem}
\newtheorem{proposition}[theorem]{Proposition}
\newtheorem{corollary}[theorem]{Corollary}
\newtheorem{lemma}[theorem]{Lemma}

\theoremstyle{definition}

\theoremstyle{remark}
\newtheorem{remark}[theorem]{Remark}

\def\sgn{\operatorname{sgn}}
\def\Pf{\operatorname{Pf}}
\def\Id{\operatorname{Id}}
\def\exc{\operatorname{exc}}
\def\inv{\operatorname{inv}}
\def\coef#1{\left\langle#1\right\rangle}

\def\la{\lambda}
\def\si{\sigma}
\def\J{\mathcal J}

\begin{document}

\title{More minor summation formulae}

\author[Chern et al.]{Shane Chern}

\author[]{Theresia Eisenk\"olbl}

\author[]{Ilse Fischer}

\author[]{Moritz Gangl}

\author[]{Mona Gatzweiler}

\author[]{\'Alvaro Guti\'errez}

\author[]{Christian Krattenthaler}

\author[]{Nishu Kumari}

\author[]{Markus Reibnegger}

\author[]{Marcus Sch\"onfelder}

\author[]{Atsuro Yoshida}

\address{\noindent
S.C., T.E., I.F., M.G., M.G., C.K., N.K., M.R., M.S., A.Y.:\newline
Fakult\"at f\"ur Mathematik, Universit\"at Wien,
Oskar-Morgenstern-Platz~1, A-1090 Vienna, Austria.\newline
\rm
\url{https://shanechern.github.io}\newline
\url{http://www.mat.univie.ac.at/~teisenko}\newline
\url{http://www.mat.univie.ac.at/~ifischer}\newline
\url{https://homepage.univie.ac.at/moritz.gangl}\newline
\url{http://www.mat.univie.ac.at/~kratt}
}

\address{A.G.: University of Bristol, Bristol, UK}

\thanks{S.C., I.F., C.K., N.K., M.R. and M.S. acknowledge support from the Austrian Science Fund (FWF) grant 10.55776/F1002.  M.G. acknowledges support from the FWF grant 10.55776/P34931.
}

\begin{abstract}
We prove determinantal-Pfaffian formulae that simultaneously
generalise the Pfaffian minor summation formula of Ishikawa and
Wakayama and Byun's recent minor summation formula. These formulae
are based on factorisation formulae for the determinant of the sum of
a skew-symmetric matrix and a rank-1 matrix. Applications include
a Cauchy-type identity for skew Schur functions.
\end{abstract}



\maketitle

\section{Introduction and statement of main results}

{\it Non-intersecting paths} are instrumental in many areas of enumerative
combinatorics, in particular in the enumeration of plane partitions
and tableaux, see e.g.\ \cite[Sec.~10.13]{KratCL} and \cite[Sec.~4]{KratCM}.
Moreover, they also appear in the guise of {\it vicious walkers}
as objects of their own interest in statistical physics,
cf.~\cite{FishAA}.

The fundamental result on the enumeration of non-intersecting paths
is Lindstr\"om's determinantal formula \cite[Theorem~1]{LindAA} 
for the number of non-intersecting paths in an acyclic directed graph which
have fixed starting and ending points. This theorem was rediscovered
several times --- and also anticipated in a more restricted form,
see \cite[Footnote~5]{KratBW}.

In several situations --- most prominently in the enumeration of
plane partitions with symmetry (cf.\ \cite[Sec.~6]{KratCM}),
it is necessary to count families of non-intersecting lattice
paths where either starting or ending points are not fixed.
In view of Lindstr\"om's determinantal formula, this leads us to
the problem of computing the sum of all maximal minors of a
rectangular matrix. Okada's {\it minor summation formula}
\cite[Theorem~3]{OkadAA} provides a Pfaffian for this sum
(cf.\ also \cite[Theorem~3.1]{StemAE}).

\begin{theorem} \label{thm:Okada}
Let $m$ and $n$ be positive integers, where $m$ is even.
Furthermore, let $A$ be an $m\times n$ matrix. Then we have
\begin{equation} \label{eq:MS} 
\underset{|I|=m}{\sum_{I\subseteq[n]}}
 \det A^I
 =
\Pf\bigl(AU_nA^t-AU_n^tA^t\bigr),
%=\Pf
% \left(
%  \sum_{1 \le r < s \le n} \left( A_{i,r} A_{j,s} - A_{i,s} A_{j,r} \right)
% \right)_{1 \le i < j \le m},
\end{equation}
where, by definition, $[n]:=\{1,2,\dots,n\}$, where
$A^I$ denotes the submatrix of $A$ consisting of the columns
indexed by~$I$,
and $U_n$ is the upper triangular matrix with all entries
above the diagonal equal to~$1$ and all other entries equal to~$0$.
\end{theorem}

This theorem can also be used in the case where $m$ is odd, that is,
for the evaluation of the sum of odd-size minors: augment the
$m\times n$ matrix~$A$ (with $m$~odd) to the $(m+1)\times(n+1)$
matrix~$\hat A$ by defining
$$
\hat A_{i,j}=\begin{cases}
A_{i,j},&\text{for $1\le i\le m$ and $1\le j\le n$},\\
1,&\text{for $i=m+1$ and $j=n+1$},\\
0,&\text{otherwise,}
\end{cases}
$$
and apply Theorem~\ref{thm:Okada} to the augmented matrix~$\hat A$.

Ishikawa and Wakayama~\cite{IsWaAA} have generalised~\eqref{eq:MS}
to the {\it Pfaffian minor summation formula}, see
Theorem~\ref{thm:IsWa} in Subsection~\ref{sec:IsWa}.

Recently, Byun~\cite[Theorem~1.3/Theorem~2.4]{ByunAA} came up with a new
formula for the sum of all maximal minors of a rectangular matrix.

\begin{theorem} \label{thm:Byun}
Let $m$ and $n$ be positive integers.
Furthermore, let $A$ be an $m\times n$ matrix. Then we have
\begin{equation} \label{eq:MS2} 
\left(\underset{|I|=m}{\sum_{I\subseteq[n]}}
 \det A^I\right)^2
 =
\det\bigl(A(2U_n+\Id_n)A^t\bigr),
\end{equation}
with the same conventions as in Theorem~\ref{thm:Okada}.
Here, and in the sequel, $\Id_n$ stands for the $n\times n$ identity
matrix.
\end{theorem}

We point out that a particularly intriguing feature of this
striking formula is the fact that it holds for even {\it and\/}
odd~$m$ alike, as opposed to Okada's formula in
Theorem~\ref{thm:Okada}.

For us it was clear that the formula~\eqref{eq:MS2} could not just
be an isolated coincidence; something more general had to be lurking
behind. Consequently, we set ourselves the goal to find the hidden, more
general identities that were waiting to be discovered.

\medskip
The purpose of this article is to report our corresponding findings.
They let us discover much more general identities, special cases of
which also cover the Pfaffian minor summation formula of Ishikawa
and Wakayama. Moreover, as by-products, we also came across
Pfaffian and determinantal identities that, as we believe, are
interesting in their own right.

\medskip
Since we also think that the path of discovery is interesting and
instructive,
we next describe how we went along and approached our goal, which in
the end let us arrive at our first main result, Theorem~\ref{thm:main1}.

Our first thought was that, with $A$ and $A^t$ appearing
in~\eqref{eq:MS2}, there should actually be {\it two} rectangular
matrices, $A$~and~$B$, in the sought-for generalisation.
It did not make sense to make $A^t$ into~$B^t$ since this does
not give anything ``nice" (whatever that means) already in the
case where $m=1$. However, in~\eqref{eq:MS2} there is this term
$2U_n$; why not split it into $U_n+U_n$? Then an ``$A$--$B$"
generalisation suggests itself: what if, instead of the right-hand
side of~\eqref{eq:MS2}, we consider
\begin{equation} \label{eq:AB1} 
\det\bigl(AU_nB^t+BU_nA^t+AB^t\bigr)\ ?
\end{equation}
We put this determinent in a computer algebra system,
and the computer told us that this
determinant decomposed into the product of two factors.
Not a square anymore, it was also not the product of one factor
depending on~$A$ and the other depending on~$B$; rather both
factors depended on both~$A$ {\it and\/}~$B$. In any case, we
managed to work out guesses for the factors, see Theorem~\ref{thm:AB} in
Subsection~\ref{sec:AB}
for the complete factorisation.

Nevertheless, the asymmetry in \eqref{eq:AB1} kept us thinking.
If we interchanged $A$ and~$B$ in~\eqref{eq:AB1}, then, according to
our previous (experimental) findings, also
\begin{equation} \label{eq:AB2} 
\det\bigl(AU_nB^t+BU_nA^t+BA^t\bigr)
\end{equation}
factorised into two factors, as predicted in Theorem~\ref{thm:AB}.
So, maybe we should interpolate between~\eqref{eq:AB1}
and~\eqref{eq:AB2}? The most straightforward way to do this
was to connect $AB^t$ and $BA^t$ by a convex combination
of the two; that is, to consider
\begin{equation} \label{eq:AB3} 
\det\bigl(AU_nB^t+BU_nA^t+\la AB^t+(1-\la)BA^t\bigr),
\end{equation}
where $\la$ is some additional variable. We went again to the
computer, which told us that this determinant decomposed into
two factors.

However, this is still asymmetric. If we transpose the matrix inside
the determinant in~\eqref{eq:AB3}, then our previous computer calculations
show that, with $L_n$ denoting the transpose of~$U_n$, also
$$
\det\bigl(AL_nB^t+BL_nA^t+\la BA^t+(1-\la)AB^t\bigr)
$$
decomposes into a product of two factors. So then, what about
the convex combination
\begin{equation} \label{eq:AB4} 
\det\Bigl(\mu\bigl(AU_nB^t+BU_nA^t\bigr)
+(1-\mu)\bigl(AL_nB^t+BL_nA^t\bigr)
+\la AB^t+(1-\la)BA^t\Bigr)\ ?
\end{equation}
The computer told us that this determinant decomposes into a
product of two factors.

It did not take us very long to realise that this was still not
``general enough". Let us replace $\mu U_n+(1-\mu)L_n+\la\Id_n$
in~\eqref{eq:AB4}
by an {\it arbitrary} $n\times n$ matrix~$X$. The computer still
said that this would factor. Leaving the path to the
identification of
the factors aside, this leads us to our first main result.

\begin{theorem} \label{thm:main1}
Let $m$ and $n$ be positive integers.
Furthermore, let $A$ and $B$ be $m\times n$ matrices,
$X$ an $n\times n$ matrix, and
$\J =(1)_{1\le i,j\le n}$. If $m$ is even, we have
\begin{equation} \label{eq:1}
\det\bigl(AXB^t+B(\J -X^t)A^t\bigr)
=f_{A,B}(X)f_{B,A}(\J -X^t)
\end{equation}
where
$$
f_{A,B}(X)=\underset {|I|=|J|=m/2}{\sum_{I,J\subseteq[n]}}
\det X_{I,J} \det(A^I\,B^J),
$$
with $X_{I,J}$ denoting the submatrix of~$X$ consisting of the entries
in the 
rows indexed by~$I$ and columns indexed by~$J$, $A^I$ denoting
the submatrix of~$A$ consisting of the columns indexed by~$I$,
with an analogous meaning of~$B^J$, and with $A^I\,B^J$ denoting
the concatenation of the matrices $A^I$ and $B^J$.

If $m$ is odd, we have
\begin{equation} \label{eq:1o}
\det\bigl(AXB^t+B(\J -X^t)A^t\bigr)
=g_{A,B}(X)g_{B,A}(\J -X^t)
=(-1)^{(m-1)/2}g_{A,B}(X)g_{B,A}(X^t),
\end{equation}
where
$$
g_{A,B}(X)=\underset {|I|=(m+1)/2,\ |J|=(m-1)/2}{\sum_{I,J\subseteq[n]}}
\det\bigl(\mathbf1 \, X_{I,J}\bigr) \det(A^I\,B^J),
$$
where $\mathbf 1$ denotes a column of length $(m+1)/2$ consisting
entirely of~$1$'s.
\end{theorem}

As we are going to show, these factorisation formulae follow from
factorisations of the determinant of the sum of a skew-symmetric
matrix and a rank-1 matrix, together with the following Pfaffian
expansion formula, which is our second main result.

\begin{theorem} \label{thm:main2}
With the notation of Theorem~\ref{thm:main1}, for even~$m$ we have
\begin{equation} \label{eq:4} 
\Pf (AXB^t-BX^tA^t)
=(-1)^{\binom {m/2}2}f_{A,B}(X).
\end{equation}
\end{theorem}

\begin{remark}
The specialisation $X=\Id _n$, that is,
$$
\Pf (AB^t-BA^t)
=(-1)^{\binom {m/2}2}f_{A,B}(\Id _n)
=(-1)^{\binom {m/2}2}
\underset{ |I|=m/2}{\sum_{I\subseteq[n]}}
\det(A^I\,B^I),
$$
can be considered as a Pfaffian analogue of the Cauchy--Binet
theorem. Okada has also given Pfaffian analogues of the Cauchy--Binet theorem,
see~\cite[Sec.~3]{OkadAB},
%$$
%\text{\tt https://www.mat.univie.ac.at/\~{}slc/wpapers/s76vortrag/okada.pdf},
%$$
but they are completely unrelated.
\end{remark}

\begin{remark}
By observing that
$$
AXB^t-BX^tA^t=(A\,B)\begin{pmatrix} 
O&X\\-X^t&O
\end{pmatrix}
(A\,B)^t,
$$
where $(A\,B)$ denotes the concatenation of the matrices~$A$ and~$B$,
one could prove Theorem~\ref{thm:main2} using the Pfaffian minor
summation formula of Ishikawa and Wakayama 
in Theorem~\ref{thm:IsWa}. However, we will
provide an independent proof in Section~\ref{sec:aux} since
we will argue in Subsection~\ref{sec:IsWa} that the special case
of Theorem~\ref{thm:main2} in which $A=B$
{\it implies} the Pfaffian minor summation formula.
In that sense, although equivalent,
Theorem~\ref{thm:main2} may be considered as a
non-symmetric extension of that formula.
\end{remark}

\medskip
Our paper is organised as follows. 
In the next section, we prove factorisation formulae for the
determinant of the sum of a skew-symmetric matrix and a rank-1 matrix,
see Propositions~\ref{prop:1} and~\ref{prop:2}. In
Section~\ref{sec:aux}, we provide the proof of
Theorem~\ref{thm:main2}, and we prove a related auxiliary result
that we need later for the proof of the odd case of
Theorem~\ref{thm:main1}. By combining the results from
Sections~\ref{sec:Y+M} and~\ref{sec:aux}, we are then in the position
to provide the proof of Theorem~\ref{thm:main1}, which we do in
Section~\ref{sec:main}. Several applications of our formulae are given
in the final section, Section~\ref{sec:appl}.
Subsection~\ref{sec:sym} contains the explicit statements of some
interesting special cases of the factorisations proved in
Section~\ref{sec:Y+M}. In Subsection~\ref{sec:IsWa} we show that
Theorem~\ref{thm:main2} implies the Pfaffian minor summation formula
of Ishikawa and Wakayama. Subsection~\ref{sec:AB} presents an asymmetric
extension of Byun's minor summation formula in
Theorem~\ref{thm:Byun} and of Okada's minor summation formula
in Theorem~\ref{thm:Okada}. 
Finally, in Subsection~\ref{sec:Cauchy} we
derive a Cauchy-type identity for skew Schur functions.

For the sake of completeness, we provide a proof of the fact that the
Pfaffian is an irreducible polynomial in the appendix. (We use this in
the proof of Theorem~\ref{thm:main1}.) This may well be hidden
somewhere in the literature, but we were not able to find an explicit
place where a proof is given.

\section{The determinant of a skew-symmetric plus a rank-$1$ matrix}
\label{sec:Y+M}

In this section, we prove two factorisations of the determinant of
the sum of a skew-symmetric matrix and a rank-1 matrix, see
Propositions~\ref{prop:1} and~\ref{prop:2} below. They will
be instrumental in the proof of Theorem~\ref{thm:main1} given in
Section~\ref{sec:main}. Moreover, as we believe, they are interesting
in their own right.

For the proofs of Propositions~\ref{prop:1} and~\ref{prop:2}, 
we heavily rely on Stembridge's combinatorial proof 
\cite[Proof of Prop.~2.2]{StemAE} of the classical fact that
the determinant of a skew-symmetric matrix is equal to the
square of the Pfaffian of the matrix. For the convenience of
the reader, we recall that proof here.

To begin with, given a skew-symmetric matrix $Y$ of size~$m$, where
$m$~is even, its Pfaffian may be defined by
$$
\Pf(Y)=\sum_{\mu\text{ a perfect matching of }[m]}
\sgn\mu\cdot
\underset{\{i,j\}\in\mu}{\prod _{1\le i<j\le m} ^{}}
Y_{i,j},
$$
where the sign $\sgn\mu$ of the perfect matching is defined
as $(-1)^{\text{cr}(\mu)}$, with $\text{cr}(\mu)$ the number
of crossings in the pictorial representation of~$\mu$; alternatively,
$\text{cr}(\mu)$ is the number of quadruples $(i,j,k,l)$ with
$i<j<k<l$ such that $\{i,k\}$ and $\{j,l\}$ are pairs in the
matching~$\mu$. 

\begin{proposition} \label{prop:0}
Let $Y$ be a skew-symmetric matrix of size~$m$, where $m$~is even.
Then
\begin{equation} \label{eq:0} 
\det(Y)=\Pf^2(Y).
\end{equation}
\end{proposition}

\begin{proof}
By definition, the determinant $\det(Y)$ equals
\begin{equation} \label{eq:detdef} 
\sum_{\si\in S_m}\sgn\si\cdot
\prod _{i=1} ^{m}Y_{i,\si(i)},
\end{equation}
where $S_m$ denotes the set of permutations of $[m]$.

Let $\si$ be a permutation that contains an odd-length cycle
in its disjoint cycle decomposition.
Among the odd-length cycles of~$\si$, 
let $c=(a_1,a_2,\dots,a_k)$ be the one that contains the smallest
element. Then form another permutation~$\si'$ by reversing the
direction of the cycle~$c$, that is, by starting with~$\si$
and subsequently replacing~$c$ by $(a_k,a_{k-1},\dots,a_1)$.
Clearly, we have $\sgn\si=\sgn\si'$. On the other hand, we have
$$
\prod _{i=1} ^{m}Y_{i,\si(i)}
=-\prod _{i=1} ^{m}Y_{i,\si'(i)}
$$
since $Y$ is skew-symmetric, and hence the reversion of the
cycle~$c$ causes a sign of $(-1)^k=-1$. As a consequence, the
contributions of~$\si$ and~$\si'$ cancel in the
sum~\eqref{eq:detdef}. Thus, only permutations~$\si$ in which all
cycles in their disjoint cycle decompositions are even contribute
to~\eqref{eq:detdef}.

Given such a permutation~$\si$ with all cycles of even length,
we may now construct\break a pair $(\mu_1,\mu_2)$ of perfect matchings
of~$[m]$ in the following way: for each cycle\break $(a_1,a_2,\dots,a_k)$
with $k$ even and $a_1$ the smallest element in the cycle, let
$\mu_1$ contain the pairings $\{a_1,a_2\}$, $\{a_3,a_4\}$, \dots,
$\{a_{k-1},a_k\}$, and let
$\mu_2$ contain the pairings $\{a_2,a_3\}$, $\{a_4,a_5\}$, \dots,
$\{a_{k},a_1\}$.
This will immediately explain the equality in~\eqref{eq:0}
once we show that 
\begin{equation} \label{eq:sgn}  
(-1)^{m-\exc(\si)}\sgn\si=(\sgn\mu_1)\cdot(\sgn\mu_2),
\end{equation}
where $\exc(\si)$ denotes the number of {\it excedances}
of~$\si$ (i.e., the number of elements~$i$ such that $\si(i)>i$).
This is done in the proof of \cite[Prop.~2.2, Eq.~(2.1)]{StemAE}, 
and it is essentially afforded by \cite[Lemma~2.1]{StemAE}.
\end{proof}

\begin{proposition} \label{prop:1}
Let $Y$ be a skew-symmetric matrix of size~$m$ and
$M=(a_ib_j)_{1\le i,j\le m}$ a rank-$1$ matrix of the same size.
If $m$ is even then 
\begin{equation} \label{eq:2} 
\det(Y+M)=\Pf(Y)\Bigg(\Pf(Y)
+\sum_{1\le i< j\le m}(-1)^{i+j-1}(a_ib_j-a_jb_i)
\Pf(Y(i,j))\Bigg),
\end{equation}
where $Y(i,j)$ arises from $Y$ by deleting the $i$-th and $j$-th
rows and columns.
\end{proposition}

\begin{proof}
We use multilinearity of the determinant in the rows to expand
$\det(Y+M)$. Since $M$ is a rank-1 matrix, the result is
$$
\det(Y+M)=\det(Y)+\sum_{i=1}^{m}\det(Y\leftarrow M_i),
$$
where $Y\leftarrow M_i$ symbolises the matrix that results from~$Y$ by
replacing 
the $i$-th row by the $i$-th row of~$M$. For example, we have
$$
(Y\leftarrow M_2)=
\begin{pmatrix}
0&Y_{1,2}&Y_{1,3}&\dots&Y_{1,m}\\
a_2b_1&a_2b_2&a_2b_3&\dots&a_2b_{m}\\
-Y_{1,3}&-Y_{2,3}&0&\dots&Y_{3,m}\\
\vdots&\vdots&\vdots&\ddots&\vdots\\
-Y_{1,m}&-Y_{2,m}&-Y_{3,m}&\dots&0
\end{pmatrix}.
$$
Now we expand each of the determinants $\det(Y\leftarrow M_i)$
along the special row, that is, the $i$-th row.
Since, in this expansion, the term $a_ib_i$ gets multiplied
by the determinant of a skew-symmetric matrix of odd size,
the corresponding term vanishes. Otherwise, we adapt the preceding proof
of Proposition~\ref{prop:0} for our purposes. Let us say
that we want to determine the cofactor of~$a_ib_j$ with $i\ne j$. 
The arguments from above can be used almost verbatim. 
The only difference is that here we must consider permutations
which map $i$ to~$j$. The involution from the previous proof cancels all
contributions from permutations that contain an odd-length cycle not
containing $i$ and~$j$. Since $m$~is even, 
if $i$ and~$j$ are contained in an odd-length cycle,
then there must be another odd-length cycle. Consequently,
contributions from such permutations are also cancelled by the
involution. In summary, only permutations survive 
which map $i$ to~$j$, and in which all cycles have even length.

Let $\si$ be such a permutation. Let
$(i,j,a_3,\dots,a_k)$, with $k$~even, be the cycle in the disjoint
cycle decomposition of~$\si$ that maps $i$ to~$j$.
From the previous proof, we know that the permutation~$\si$
can be mapped to a pair $(\mu_1,\mu_2)$ of perfect matchings
with the property that
\begin{equation*} %\label{eq:sgn} 
(-1)^{m-\exc(\si)}\sgn\si=(\sgn\mu_1)\cdot(\sgn\mu_2)
\end{equation*} 
by --- essentially --- ``cutting" each cycle of~$\si$ into
two matchings. One of~$\mu_1$ and~$\mu_2$ will contain the
pairings $\{i,j\}$, $\{a_3,a_4\}$, \dots, $\{a_{k-1},a_k\}$,
while the other contains the pairings $\{a_k,i\}$, $\{j,a_3\}$, \dots.
By convention, let $\mu_2$ contain $\{i,j\}$.
Then $\mu_1$ contributes to $\Pf(Y)$ on the right-hand side
of~\eqref{eq:2}, while $\mu_2\backslash\{\{i,j\}\}$ contributes
to~$\Pf(Y(i,j))$. Since we know that~\eqref{eq:sgn} holds, it remains
to determine the effect of removing $\{i,j\}$ from~$\mu_2$.
Indeed, the number of crossings of~$\mu_2$ in which $\{i,j\}$
is involved is congruent to $i-j-1$ modulo~2. Thus, if $i<j$
there is a sign of $(-1)^{i+j-1}$ which has to be recorded,
while if $i>j$ it is $-(-1)^{i+j-1}$ (due to lack of skew-symmetry
of $a_ib_j$ versus~$a_jb_i$).

This explains the claimed result.
\end{proof}

%\begin{remark}
%\cite{BrillAA}
%\end{remark}

\begin{proposition} \label{prop:2}
Let $Y$ be a skew-symmetric matrix of size~$m$ and
$M=(a_ib_j)_{1\le i,j\le m}$ a rank-$1$ matrix of the same size.
If $m$ is odd then 
\begin{equation} \label{eq:5} 
\det(Y+M)=\Bigg(
\sum_{i=1}^m(-1)^{i-1}a_i
\Pf(Y(i))\Bigg)\Bigg(
\sum_{j=1}^m(-1)^{j-1}b_j
\Pf(Y(j))\Bigg),
\end{equation}
where $Y(i)$ arises from $Y$ by deleting the $i$-th 
row and column.
\end{proposition}

\begin{proof}
We proceed as in the proof of Proposition~\ref{prop:1}. That is,
we start by using multilinearity of the determinant in the rows to expand
$\det(Y+M)$. Here we obtain
$$
\det(Y+M)=\sum_{i=1}^{m}\det(Y\leftarrow M_i),
$$
where $Y\leftarrow M_i$ has the same meaning as before.
Now we expand each of the determinants $\det(Y\leftarrow M_i)$
along the special row, that is, the $i$-th row.
Let us say
that we want to determine the cofactor of~$a_ib_j$. 
If $i=j$, then we must show that the determinant of~$Y(i)$
equals the square of the Pfaffian of~$Y(i)$.
This is indeed the case because of Proposition~\ref{prop:0}.

From now on we assume $i\ne j$. In order to determine the cofactor
of~$a_ib_j$ in $Y\leftarrow M_i$ in the spirit of the proof of
Proposition~\ref{prop:1},
we must again consider permutations
which map $i$ to~$j$. The involution from the proof of
Proposition~\ref{prop:0} cancels all
contributions from permutations that contain an odd-length cycle not
containing $i$ and~$j$. With $m$ being odd, we see that
only contributions from permutations survive 
which map $i$ to~$j$, in which $i$ and~$j$ are contained in an
odd-length cycle, and in which all other cycles have even length.
Let $\si$ be such a permutation.
Let
$(j,a_3,\dots,a_k,i)$, with $k$~odd, be the cycle in the disjoint
cycle decomposition of~$\si$ that maps $i$ to~$j$.
Inspired by the proof of Proposition~\ref{prop:0}, we map $\si$
to the pair of matchings $(\mu_1,\mu_2)$, where for each even-length
cycle we do the decomposition as in that proof, and where we
stipulate that the pairings $\{j,a_3\}$, $\{a_4,a_5\}$, \dots,
$\{a_{k-1},a_k\}$ are put into~$\mu_1$, and
the pairings $\{a_3,a_4\}$, \dots, $\{a_{k-2},a_{k-1}\}$, 
$\{a_{k},i\}$ are put into~$\mu_2$. Clearly, the matching~$\mu_1$
is a perfect matching on $[m]\backslash\{i\}$, while $\mu_2$
is a perfect matching on $[m]\backslash\{j\}$.

Instead of $(\si,\mu_1,\mu_2)$, let us for the moment consider
$(\si',\mu'_1,\mu_2)$, where $\si'$ arises from~$\si$ by
replacing the cycle $(j,a_3,\dots,a_k,i)$ by $(a_3,\dots,a_k,i)$,
and where $\mu_1'$ arises from $\mu_1$ by replacing the element~$j$
by~$i$. Then the matchings~$\mu_1'$ and~$\mu_2$ arise from the
permutation~$\si'$ by a decomposition process as
in the proof of Proposition~\ref{prop:0}, respectively by
a slight modification
thereof (when decomposing the special cycle $(a_3,\dots,a_k,i)$ we followed
a different rule). In any case, 
the permutation $\si'$ --- which is a permutation on
$[m]\backslash\{j\}$ --- and the pair of matchings~$(\mu_1',\mu_2)$
--- the latter being perfect matchings on $[m]\backslash\{j\}$,
still satisfy the sign relation~\eqref{eq:sgn}, that is,
\begin{equation} \label{eq:sgn2} 
(-1)^{m-1-\exc(\si')}\sgn\si'=(\sgn\mu_1')\cdot(\sgn\mu_2).
\end{equation} 
As is easy to see, we have 
\begin{align*}
\sgn\si'&=-\sgn\si,\\
\exc\si'&=\exc\si-\chi(i<j)-\chi(j<a_3)+\chi(i<a_3),\\
\sgn\mu_1'&=(-1)^{i-j-1-\chi(a_3\text{ is between }i\text{ and }j)}\sgn\mu_1,
\end{align*}
where, by definition, $\chi(\mathcal S)=1$ if $\mathcal S$ is true and
$\chi(\mathcal S)=0$ otherwise. Since
$$-\chi(j<a_3)+\chi(i<a_3)\equiv
\chi(a_3\text{ is between }i\text{ and }j)\pmod2,$$
as a verification over all the possible six relative orders of $i,j,a_3$
shows,
the sign relation~\eqref{eq:sgn2} 
implies
\begin{equation*} %\label{eq:sgn2} 
(-1)^{m-\exc\si+\chi(i>j)}\sgn\si=(-1)^{i+j}(\sgn\mu_1)\cdot(\sgn\mu_2).
\end{equation*} 
Clearly, the matching $\mu_1$ contributes to $\Pf(Y(i))$
on the right-hand side
of~\eqref{eq:5}, while $\mu_2$ contributes
to~$\Pf(Y(j))$. 

This explains the claimed result, with the term $\chi(i>j)$ in the
exponent on the left-hand side taking care that the $(i,j)$-entry
of $Y\leftarrow M_i$ --- that is, $a_ib_j$ --- does not contribute a sign.
\end{proof}

\section{Proof of Theorem~\ref{thm:main2}, and an auxiliary result} 
\label{sec:aux}

\begin{proof}[Proof of Theorem~\ref{thm:main2}]
The Pfaffian is multilinear in the rows of $\bigl(A\mid B\bigr)$.
The same applies to $f_{A,B}(X)$. Hence it suffices to check
the identity for matrices $\bigl(A\mid B\bigr)$ that contain at
most one~1 in each row and otherwise zeroes.
If $\bigl(A\mid B\bigr)$ contains a row of 0's then
both sides of~\eqref{eq:4} vanish, and the same is true if two
rows of $\bigl(A\mid B\bigr)$ are the same.
On the other hand, if $\bigl(A\mid B\bigr)$
has exactly one~1 in each row, then we may reorder the rows so that
the left-most~1 is in the top row, the left-most~1 of the remaining
rows is in the second row, etc. The reordering of the rows causes
the same sign on both sides of~\eqref{eq:4}, due to 
\cite[Prop.~2.3(b)]{StemAE}. 
In summary, we may restrict our attention to
$$
\bigl(A\mid B\bigr)=\begin{pmatrix}
\mathcal I^R&O\\
O&\mathcal I^S\end{pmatrix},
$$
where $R,S\subseteq[n]$ and $|R|+|S|=m$, and
given $R=\{r_1,r_2,\dots\}$ with $r_1<r_2<\cdots$ the matrix $\mathcal
I^R$ is
an $|R|\times n$ matrix consisting of 0's and 1's with the unique~1 in row~$i$
located in column~$r_i$. On the left-hand side of~\eqref{eq:4}, we obtain
$$
\Pf\begin{pmatrix} O&X_{R,S}\\-X_{R,S}^t&O\end{pmatrix}
=\begin{cases}
(-1)^{\binom {m/2}2}\det(X_{R,S}),&\text{if }|R|=|S|=m/2,\\
0,&\text{otherwise.}
\end{cases}
$$
On the right-hand side of \eqref{eq:4}, we obtain the same.
\end{proof}

\begin{lemma} \label{lem:2}
With the notation of Theorem~\ref{thm:main1}, let
$Y=AXB^t-BX^tA^t$.
For odd~$m$ we have
\begin{equation} \label{eq:6} 
\sum_{i=1}^m(-1)^{i-1}\left(\sum_{j=1}^nA_{i,j}\right)\Pf (Y(i))
=(-1)^{\binom {(m-1)/2}2}g_{A,B}(X).
\end{equation}
\end{lemma}

\begin{proof}
This proof parallels the proof of Theorem~\ref{thm:main2} given just
above.

The left-hand side of \eqref{eq:6} 
is multilinear in the rows of $\bigl(A\mid B\bigr)$.
The same applies to $g_{A,B}(X)$. Hence it suffices to check
the identity for matrices $\bigl(A\mid B\bigr)$ that contain at
most one~1 in each row and otherwise zeroes.
If $\bigl(A\mid B\bigr)$ contains a row of 0's then
both sides of~\eqref{eq:6} vanish, and the same is true if two
rows of $\bigl(A\mid B\bigr)$ are the same.
On the other hand, if $\bigl(A\mid B\bigr)$
has exactly one~1 in each row, then we may reorder the rows so that
the left-most~1 is in the top row, the left-most~1 of the remaining
rows is in the second row, etc.

We claim that
the reordering of the rows causes
the same sign on both sides of~\eqref{eq:6}.
Indeed, let us suppose that we interchange row~$r$ and row~$s$
of $\bigl(A\mid B\bigr)$. On the right-hand side this causes a
switch of sign. The same is true for the term $\Pf(Y(i))$ on the
left-hand side as long as $i\notin\{r,s\}$, due to 
\cite[Prop.~2.3(b)]{StemAE}. It remains to consider the terms
for $i\in\{r,s\}$, that is,
\begin{equation} \label{eq:Yrs} 
(-1)^{r-1}\left(\sum_{j=1}^nA_{s,j}\right)\Pf
\left(Y(r)\big\vert_{s\to r}\right)
+(-1)^{s-1}\left(\sum_{j=1}^nA_{r,j}\right)\Pf
\left(Y(s)\big\vert_{r\to s}\right).
\end{equation}
Here, $Y(r)\big\vert_{s\to r}$ denotes the skew-symmetric matrix
arising from~$Y$ by deleting the $r$-th row and column, and by
replacing the $s$-th row and column by the $r$-th row and column
of~$Y$, with an analogous meaning for $Y(s)\big\vert_{r\to s}$. 
Now, the matrix $Y(r)\big\vert_{s\to r}$ is essentially the same as
$Y(s)$, except that row~$r$ and column~$r$ are not in the right place.
By doing the appropriate simultaneous interchanges of rows and
columns, we see that
$\Pf\left(Y(r)\big\vert_{s\to r}\right)=(-1)^{r-s-1}\Pf(Y(s))$,
with an analogous equality for the other Pfaffian. Hence, the
expression in~\eqref{eq:Yrs} equals
$$
(-1)^{s}\left(\sum_{j=1}^nA_{s,j}\right)\Pf(Y(s))
+(-1)^{r}\left(\sum_{j=1}^nA_{r,j}\right)\Pf(Y(r)),
$$
which is effectively the negative of the corresponding terms on the
left-hand side of~\eqref{eq:6}.

In summary, we may restrict our attention to
$$
\bigl(A\mid B\bigr)=\begin{pmatrix}
\mathcal I^R&O\\
O&\mathcal I^S\end{pmatrix},
$$
where $R,S\subseteq[n]$ and $|R|+|S|=m$, and, as before,
given $R=\{r_1,r_2,\dots\}$ with $r_1<r_2<\cdots$ the matrix $\mathcal
I^R$ is
an $|R|\times n$ matrix consisting of 0's and 1's with the unique~1 in row~$i$
located in column~$r_i$. On the left-hand side of~\eqref{eq:6}, we obtain
\begin{multline*}
\sum_{i\ge1}(-1)^{i-1}
\Pf\begin{pmatrix} O&X_{R\backslash\{r_i\},S}\\
-X_{R\backslash\{r_i\},S}^t&O\end{pmatrix}\\
=\begin{cases}
\displaystyle(-1)^{\binom {(m-1)/2}2}\sum_{i=1}^{(m+1)/2}(-1)^{i-1}
\det(X_{R\backslash\{r_i\},S}),&\text{if $|R|=(m+1)/2$
and $|S|=(m-1)/2$},\\
0,&\text{otherwise.}
\end{cases}
\end{multline*}
By Laplace expansion, the expression in the first alternative on the
right-hand side simplifies to $(-1)^{\binom {(m-1)/2}2}
\det(\mathbf1\,X_{R,S})$.
On the right-hand side of \eqref{eq:6}, we obtain the same result.
\end{proof}

\section{Proof of Theorem~\ref{thm:main1}}
\label{sec:main}

\begin{proof}[Proof of \eqref{eq:1}]
First let $m\le n$. We begin with the observation that
$$AXB^t-BX^tA^t=AXB^t-(AXB^t)^t$$
is a skew-symmetric matrix and that
$B\J A^t$ is a rank-1 matrix. Hence,
using Proposition~\ref{prop:1}, we obtain
$$
\det\bigl(AXB^t+B(\J -X^t)A^t\bigr)
=\Pf(AXB^t-BX^tA^t)\text{Expr}_1(A,X,B),
$$
where $\text{Expr}_1(A,X,B)$ is the result of the substitution
$Y=AXB^t-BX^tA^t$ and $M=B\J A^t$ in the second factor on the right-hand
side of~\eqref{eq:2}. On the other hand, again by Proposition~\ref{prop:1}, 
we also have
\begin{align*}
\det\bigl(AXB^t+B(\J -X^t)A^t\bigr)
&=
\det\bigl(B(\J -X^t)A^t-A(\J -X)B^t+A\J B^t\bigr)\\
&=
\Pf\bigl(B(\J -X^t)A^t-A(\J -X)B^t\bigr)\text{Expr}_2(A,X,B),
\end{align*}
where $\text{Expr}_2(A,X,B)$ is the result of the substitution
$Y=B(\J -X^t)A^t-A(\J -X)B^t$ and $M=A\J B^t$ in the second factor on the
right-hand side of~\eqref{eq:2}.

Combining the last two factorisations, we see that both
$\Pf(AXB^t-BX^tA^t)$ and $\Pf\bigl(B(\J -X^t)A^t-A(\J -X)B^t\bigr)$ are factors
of the determinant on the left-hand side of~\eqref{eq:1}. From 
Theorem~\ref{thm:main2} we know
that these are, up to sign, $f_{A,B}(X)$ and $f_{B,A}(\J -X^t)$. By
Lemma~\ref{lem:3}, both are
irreducible as polynomials in their variables.\footnote{The precise
argument is the following: we choose $A_0$ as the matrix which has
$a_i$ as entry in the $i$-th row and column, $i=1,2,\dots,m$, and
otherwise all entries equal zero. Similarly,
we choose $B_0$ as the matrix which has
$b_i$ as entry in the $i$-th row and column, $i=1,2,\dots,m$,
all other entries being equal to zero.
Then $A_0XB_0^t-B_0X^tA_0^t=(a_iX_{i,j}b_j-a_jX_{j,i}b_i)_{1\le i,j\le m}$.
We claim that the Pfaffian of this matrix, $C$ say, is an irreducible
polynomial in its variables. To see this, first put all $a_i$'s
and~$b_i$'s equal to~1. Then
$C$ reduces to 
\begin{equation} \label{eq:Xij} 
(X_{i,j}-X_{j,i})_{1\le i,j\le m}.
\end{equation}
Lemma~\ref{lem:3} implies that its Pfaffian is an irreducible polynomial in its
variables~$X_{i,j}$. Indeed, it is irreducible if we put all $X_{i,j}$'s
with $i>j$ to zero, and it must remain irreducible when we remove this
assignment because of degree reasons.
The Pfaffian of~\eqref{eq:Xij} must remain irreducible if we replace
$X_{i,j}$ by $a_iX_{i,j}b_j$ for all $i$ and~$j$.
Hence, the Pfaffian of the matrix~$C$ is indeed an
irreducible polynomial.
This being the case, the Pfaffian of the more general
matrix $AXB^t-BX^tA^t$ must also be irreducible.}
Moreover, they are not equal or proportional to each other.
Hence, due to degree considerations, the two sides of~\eqref{eq:1} are
equal to each other, up to some multiplicative constant.
In order to determine that constant, we choose $A_0=B_0=(\Id_m\,O)$.
With that choice, the left-hand side of~\eqref{eq:1} equals
$$\det\bigl(A_0XB_0^t+B_0(\J-X^t)A_0^t\bigr)=
\det(X_{[m],[m]}-X_{[m],[m]}^t+\J_{[m],[m]}),$$
while we have
$$
f_{A_0,B_0}(X)=(-1)^{\binom {m/2}2}\Pf(A_0XB_0^t-B_0X^tA_0^t)
=(-1)^{\binom {m/2}2}\Pf(X_{[m],[m]}-X_{[m],[m]}^t)
$$
and
\begin{align*}
f_{B_0,A_0}(\J-X^t)&=(-1)^{\binom {m/2}2}
\Pf\bigl(B_0(\J-X^t)A_0^t-A_0(\J-X)B_0^t\bigr)\\
&=(-1)^{\binom {m/2}2}\Pf(X_{[m],[m]}-X_{[m],[m]}^t).
\end{align*}
By Proposition~\ref{prop:1} with $a_i=b_j=1$ for all~$i$ and~$j$,
we see that the sought-for multiplicative constant equals~1.

\medskip
Finally we address the case where $m>n$. Define the $m\times m$
matrices $A'$ and $B'$ by $A'=(A\mid 0)$ and $B'=(B\mid 0)$, that is,
$A'$ arises from~$A$ by appending $m-n$ zero columns on the right,
and similarly for~$B'$. Furthermore, define the $m\times m$ matrix
$X'$ by $X'=(X_{i,j})_{1\le i,j\le m}$. We use Theorem~\ref{thm:main1}
with $A',B',X'$ in place of~$A,B,X$. This is allowed since we have
just proved the validity of Theorem~\ref{thm:main1} for such a
case. It is now easy to see that this
specialisation of the theorem reduces to Theorem~\ref{thm:main1}
with the $m\times n$ matrices $A$ and~$B$ and the $n\times n$
matrix~$X$. 

\medskip
This finishes the proof of~\eqref{eq:1}.
\end{proof}

\begin{proof}[Proof of \eqref{eq:1o}]
We proceed in complete analogy to the previous proof.
Also here, we observe that
$$AXB^t-BX^tA^t=AXB^t-(AXB^t)^t$$
is a skew-symmetric matrix and that
$B\J A^t=(b_ia_j)_{1\le i,j\le m}$ is a rank-1 matrix with
$b_i=\sum_{\ell=1}^nB_{i,\ell}$ and
$a_j=\sum_{\ell=1}^nA_{j,\ell}$.
Hence,
using Proposition~\ref{prop:2} with $Y=AXB^t-BX^tA^t$
and $M=B\J A^t$, we obtain
\begin{equation} \label{eq:7}
\det\bigl(AXB^t+B(\J -X^t)A^t\bigr)
=\Bigg(
\sum_{i=1}^m(-1)^{i-1}a_i
\Pf(Y(i))\Bigg)\Bigg(
\sum_{i=1}^m(-1)^{i-1}b_i
\Pf(Y(i))\Bigg).
\end{equation}
By Lemma~\ref{lem:2}, we see that the first factor equals
$(-1)^{\binom {(m-1)/2}2}g_{A,B}(X)$. Similarly, by interchanging the
roles of~$A$ and~$B$ and replacing~$X$ by~$X^t$ in~\eqref{eq:6},
we see that the second factor equals
$(-1)^{(m-1)/2}(-1)^{\binom {(m-1)/2}2}g_{B,A}(X^t)$. This yields the
expression on the far right in~\eqref{eq:1o}.
On the other hand, if we apply the argument of the previous
proof of~\eqref{eq:1} to~\eqref{eq:7}, then we obtain the alternative
in the centre of~\eqref{eq:1o}.
\end{proof}

\section{Applications}
\label{sec:appl}

Here we present applications of our main results in
Theorems~\ref{thm:main1} and~\ref{thm:main2}, and of the factorisation
results in Section~\ref{sec:Y+M}.

\subsection{The ``symmetric" cases of the factorisation identities from
Section~\ref{sec:Y+M}}
\label{sec:sym}

In this subsection, we state the special cases of
Propositions~\ref{prop:1} and~\ref{prop:2} where $a_i=b_i$
for all~$i$ explicitly. 
A remarkable feature is that in both cases the right-hand sides
are perfect squares. Since, as we shall see in
Subsection~\ref{sec:AB}, Byun's minor summation
formula~\eqref{eq:MS2} is a special
case of this ``symmetric" situation, this provides the explanation for
the square on the left-hand side of Byun's formula.

\begin{corollary} \label{cor:5}
Let $Y$ be a skew-symmetric $m\times m$ matrix and
$M=(a_ia_j)_{1\le i,j\le m}$ a symmetric
rank-$1$ matrix of the same size.
If $m$ is even then
$$
\det(Y+M)=\Pf^2(Y)=\det(Y).
$$

If $m$ is odd then
$$
\det(Y+M)=\Bigg(
\sum_{i=1}^m(-1)^{i-1}a_i      
\Pf(Y(i))\Bigg)^2,
$$
where $Y(i)$ arises from $Y$ by deleting the $i$-th
row and column.
\end{corollary}

If we apply Corollary~\ref{cor:5} to
Theorem~\ref{thm:main1} with $A=B$, then we obtain
the following identity.

\begin{corollary} \label{cor:7}
Let $A$ be an $m\times n$ matrix,
$X$ an $n\times n$ matrix, and
$\J =(1)_{1\le i,j\le n}$. If $m$ is even, we have
$$
\det\bigl(A(X+\J -X^t)A^t\bigr)      
=\det\bigl(A(X-X^t)A^t\bigr)
=\bigl(f_{A}(X)\bigr)^2,
$$
where
$$
f_{A}(X)=\underset {|I|=|J|=m/2}{\sum_{I,J\subseteq[n]}}
\det X_{I,J} \det(A^IA^J),
$$
with $X_{I,J}$ denoting the submatrix of~$X$ consisting of the entries
in the
rows indexed by~$I$ and columns indexed by~$J$, and $A^I$ denoting
the submatrix of~$A$ consisting of the columns indexed by~$I$.
\end{corollary}

\subsection{The Pfaffian minor summation formula of Ishikawa and
Wakayama}
\label{sec:IsWa}

The purpose of this subsection is to demonstrate that
Theorem~\ref{thm:main2} implies the Pfaffian minor summation formula
of Ishikawa and Wakayama~\cite[Theorem~1 with $q=1$]{IsWaAA}.

\begin{theorem} \label{thm:IsWa}
Let $m$ and $n$ be positive integers, where $m$ is even.
Furthermore, let $A$ be an $m\times n$ matrix, and let
$Y=(Y_{i,j})_{1\le i,j\le n}$ be a
skew-symmetric matrix. Then, with the notation of
Theorem~\ref{thm:main1}, we have
\begin{equation} \label{eq:IsWa} 
\underset{|I|=m}{\sum_{I\subseteq[n]}}
\Pf(Y_{I,I}) \det (A^I)
 =
\Pf\bigl(AYA^t\bigr).
\end{equation}
\end{theorem}

\begin{proof}
Let $X$ be the $n\times n$ upper triangular matrix with 0's on the
main diagonal and $X_{i,j}=Y_{i,j}$ for $1\le i<j\le n$.
Then we have $Y=X-X^t$.
Setting $A=B$ in Theorem~\ref{thm:main2}, we get 
$$
\Pf (AYA^t)=
\Pf (AXA^t-AX^tA^t)
=(-1)^{\binom {m/2}2}
\underset {|J|=|K|=m/2}{\sum_{J,K\subseteq[n]}}
\det (X_{J,K}) \det(A^JA^K).
$$
Clearly, the second determinant in the summand
on the right-hand side vanishes if
$J$ and~$K$ have a non-trivial intersection. Hence, we may rewrite the
above equation as
$$
\Pf (AYA^t)
=(-1)^{\binom {m/2}2}
\underset {|I|=m}{\sum_{I\subseteq[n]}}
\det(A^I)\underset{|J|=|K|=m/2}{\sum_{J\cup K=I}}
(-1)^{\inv(JK)}\det (X_{J,K} ),
$$
where $\inv(JK)$ denotes the number of inversions in the word $JK$,
that is, the number of pairs $(j,k)$ with $j\in J$, $k\in K$, and
$j>k$. The identity~\eqref{eq:IsWa} now follows from
Lemma~\ref{lem:IsWa} below.
\end{proof}

\begin{lemma} \label{lem:IsWa}
Let $m$ be an even positive integer and $I$ a subset of~$[n]$
with $|I|=m$. Then, with the notation from 
Theorem~\ref{thm:IsWa} and its proof, we have
\begin{equation} \label{eq:inv} 
\underset{|J|=|K|=m/2}{\sum_{J\cup K=I}}
(-1)^{\binom {m/2}2+\inv(JK)}\det (X_{J,K} )
=\Pf(Y_{I,I}).
\end{equation}
\end{lemma}

\begin{proof}
Without loss of generality, we may assume that $I=[m]$.

We prove \eqref{eq:inv} by induction on~$m$. If $m=2$, then on the
left-hand side (with $I=\{1,2\}$) there is only one choice for
$J$ and~$K$, namely $J=\{1\}$ and $K=\{2\}$.
(The reader must recall that $X$ is an upper triangular matrix.)
Hence, this left-hand
side becomes $(-1)^{\binom 12+\inv(12)}\det(Y_{1,2})=Y_{1,2}$.
On the right-hand side we obtain the same.

For the induction step, we first observe that $K$ must contain~$m$
because otherwise the last row of $X_{J,K}$ would entirely
consist of~0's so that its determinant vanishes.
In a determinant $\det(X_{J,K})$ we focus on terms that
contain the variable $Y_{j,m}$, for a fixed~$j$. Roughly speaking, we
would like to ``remove" $Y_{j,m}$ from the expression and then apply
the induction hypothesis to whatever remains.
So, let $J'=J\backslash\{j\}$ and $K'=K\backslash\{m\}$.
We then have 
$$\inv(JK)=\inv(J'K')+|K\cap[1,j-1]|$$
and the coefficient of $Y_{j,m}$ in the determinant $\det(
X_{J,K})$ equals
$$
\coef{Y_{j,m}}\det(X_{J,K})
=(-1)^{|J\cap[j+1,m]|}\det(X_{J',K'}),
$$
which is seen by expansion along the last column.
Since we have
\begin{align*}
|J\cap[j+1,m]|&=|([m]\backslash K)\cap[j+1,m]|\\
&=|[j+1,m]\backslash (K\cap[j+1,m])|\\
&=m-j-|K\cap[j+1,m]|\\
&=m-j-\bigl(\tfrac {m} {2}-|K\cap[1,j-1]|\bigr)\\
&=\tfrac {m} {2}-j+|K\cap[1,j-1]|,
\end{align*}
we obtain for the left-hand side of \eqref{eq:inv} (with $I=[m]$)
\begin{multline*}
\sum_{j=1}^{m-1}Y_{j,m}
\underset{|J'|=|K'|=(m-2)/2}{\sum_{J'\cup K'=[m-1]\backslash\{j\}}}
(-1)^{\binom {m/2}2+\inv(J'K')
+|K\cap[1,j-1]|+\frac {m} {2}-j+|K\cap[1,j-1]|}\det (X_{J',K'}) 
\\=
\sum_{j=1}^{m-1}(-1)^{m-j-1}Y_{j,m}
\underset{|J'|=|K'|=(m-2)/2}{\sum_{J'\cup K'=[m-1]\backslash\{j\}}}
(-1)^{\binom {(m-2)/2}2+\inv(J'K')}\det (X_{J',K'}) .
\end{multline*}
By induction hypothesis, this reduces to
$$
\sum_{j=1}^{m-1}(-1)^{m-j-1}Y_{j,m}
\Pf(Y_{[m-1]\backslash\{j\},[m-1]\backslash\{j\}})
=\Pf(Y_{[m],[m]}),
$$
by Laplace expansion of Pfaffians. This establishes \eqref{eq:inv}.
\end{proof}

\subsection{Asymmetric extensions of the minor summation formulae of
Byun and of Okada}
\label{sec:AB}

The goal of this subsection is to show that Byun's minor summation
formula~\eqref{eq:MS2} is implied by Theorem~\ref{thm:main1}, and,
more generally, to derive an extension of Byun's formula that
involves {\it two} rectangular matrices $A$ and~$B$, see
Theorem~\ref{thm:AB} below. Moreover, we also get an extension of
Okada's minor summation formula~\eqref{eq:MS} for free; see
Theorem~\ref{thm:AB2} below.
However, before we are able to derive these results, we need to show 
two auxiliary results.

\begin{lemma} \label{lem:X1}
Let $n$ and $\ell$ be positive integers.
Let $X$ be an upper triangular $n\times n$ matrix with $1$'s above
the diagonal, that is, $X_{i,j}=1$ for $1\le i<j\le n$.
Furthermore let $I=\{i_1,i_2,\dots,i_\ell\}$ and~$J=\{j_1,j_2,\dots,j_\ell\}$
be subsets of~$[n]$ of equal cardinality, both of which we assume to
be ordered.
Then, with the notation of Theorem~\ref{thm:main1}, we have
\begin{equation} \label{eq:X1} 
\det(X_{I,J})=\begin{cases}
\prod _{k=1} ^{\ell}X_{i_k,i_k}^{\chi(i_k=j_k)}
(1-X_{i_k,i_k})^{\chi(j_{k-1}=i_k)},&
\text{if }1\le i_1\le j_1\le i_2\le\dots\le j_\ell\le n,\\
0,&\text{otherwise}.
\end{cases}
\end{equation}
Here, by convention, $j_0=0$, and, as earlier,
$\chi(\mathcal S)=1$ if $\mathcal S$ is true and
$\chi(\mathcal S)=0$ otherwise.
\end{lemma} 

\begin{proof}
We prove the claimed identity by induction on~$\ell$.

The claim is trivially true for $\ell=0$ (if we interpret both the
determinant of a $0\times0$ matrix and the empty product as~1)
and for $\ell=1$. Therefore, from now on, we assume that $\ell\ge2$.

For the induction step, we distinguish several cases.

\begin{itemize}
\item If $i_\ell= j_\ell$, then the last row of~$X_{I,J}$ 
contains only one non-zero element,
and we get $\det(X_{I,J}) = X_{ i_\ell ,i_\ell}
\det X_{ I\backslash
\{ i_\ell \}, J\backslash \{ j_\ell \} }$.
Since $j_{\ell-1} < j_\ell = i_\ell$,
the result is true by induction hypothesis. 
\item
If $i_\ell < j_\ell$ and $j_{\ell-1} > i_\ell $, then the last two
columns of the minor are all~1's and the determinant vanishes, as
desired.
\item
If $i_\ell < j_\ell$ and $j_{\ell-1} < i_\ell $, then the last row
contains only one non-zero element, namely a~1, and we get $\det(X_{I,J}) =
\det X_{I\backslash \{ i_\ell \}, J\backslash \{ j_\ell \} }$.
Since $j_{\ell-1} < i_\ell $, the
result is true by induction hypothesis. 
\item
If $i_\ell < j_\ell$ and $j_{\ell-1} = i_\ell $, then the last column
only contains 1's, while the second-to-last column contains $\ell-1$
entries~1, but ends with $X_{ i_\ell, i_\ell }$. We subtract the
second-to-last column from the last column. Then the last column is
all 0's except for the last entry $1 - X_{ i_\ell ,i_\ell }$. This gives
$\det(X_{I,J}) = (1 - X_{ i_\ell, i_\ell })
\det X_{ I\backslash \{ i_\ell \}, J\backslash \{ j
_\ell \}}$. Since the equality at $j_{\ell-1} = i_\ell$
has already been
taken into account, the result is true by induction hypothesis. 
\item
If $i_\ell > j_\ell$ then the last row of $X_{I,J}$ is a row of 0's.
Hence its determinant is zero.
\end{itemize}
This completes the proof of the lemma.
\end{proof}

\begin{lemma} \label{lem:X2}
Let $n$ and $\ell$ be positive integers.
Let $X$ be an upper triangular $n\times n$ matrix with $1$'s above
the diagonal, that is, $X_{i,j}=1$ for $1\le i<j\le n$.
Furthermore let $I=\{i_1,i_2,\dots,i_{\ell+1}\}$
and~$J=\{j_1,j_2,\dots,j_\ell\}$
be subsets of\/~$[n]$, both of which we assume to
be ordered.
Then, with the notation of Theorem~\ref{thm:main1}, we have
\begin{equation} \label{eq:X2} 
\det(\mathbf 1\,X_{I,J})=\begin{cases}
(-1)^\ell\prod _{k=1} ^{\ell}X_{i_k,i_k}^{\chi(i_k=j_k)}
(1-X_{i_{k+1},i_{k+1}})^{\chi(j_{k}=i_{k+1})},&\\
&\kern-2cm
\text{if }1\le i_1\le j_1\le i_2\le\dots\le j_\ell\le i_{\ell+1}\le n,\\
0,&\kern-2cm\text{otherwise}.
\end{cases}
\end{equation}
Here, $\mathbf1$ denotes a column of $\ell+1$ elements~$1$.
\end{lemma} 

\begin{proof}
We prove the claimed identity by induction on~$\ell$, also using
the result from Lemma~\ref{lem:X1}.

The claim is trivially true for $\ell=0$ (if we interpret the empty
product as~1)
and for $\ell=1$. Therefore, from now on, we assume that $\ell\ge2$.

For the induction step, we distinguish several cases.

\begin{itemize}
\item If $i_{\ell+1}= j_\ell$, then the last row of~$X_{I,J}$ 
contains only one non-zero element, namely
$X_{i_{\ell+1},j_\ell}=X_{i_{\ell+1},i_{\ell+1}}$
as the last entry in the row. By expansion along the last row,
we get
\begin{equation} \label{eq:8} 
\det(\mathbf1\,X_{I,J}) =
(-1)^{\ell}\det(X_{I\backslash\{i_{\ell+1}\},J})+
X_{ i_{\ell+1} ,i_{\ell+1}}
\det(\mathbf1\, X_{ I\backslash
\{ i_{\ell+1} \}, J\backslash \{ j_\ell \} }).
\end{equation}
We have $i_{\ell} <  i_{\ell+1}=j_\ell$. Hence,
by Lemma~\ref{lem:X1} and the induction hypothesis, the right-hand
side can only be non-zero if
$1\le i_1\le j_1\le i_2\le\dots\le j_{\ell-1}\le i_\ell<j_\ell\le n$.
Assuming this inequality chain, by Lemma~\ref{lem:X1} and
by induction hypothesis, we obtain
\begin{align*}
\det(\mathbf1\,X_{I,J}) &=(-1)^{\ell}
\prod _{k=1} ^{\ell}X_{i_k,i_k}^{\chi(i_k=j_k)}
(1-X_{i_k,i_k})^{\chi(j_{k-1}=i_k)}\\
&\kern1cm
+
X_{ i_{\ell+1} ,i_{\ell+1}}
(-1)^{\ell-1}\prod _{k=1} ^{\ell-1}X_{i_k,i_k}^{\chi(i_k=j_k)}
(1-X_{i_{k+1},i_{k+1}})^{\chi(j_{k}=i_{k+1})}\\
&=(-1)^{\ell}(1-X_{i_{\ell+1},i_{\ell+1}})
\prod _{k=1} ^{\ell}X_{i_k,i_k}^{\chi(i_k=j_k)}
(1-X_{i_{k},i_{k}})^{\chi(j_{k-1}=i_{k})}.
\end{align*}
This is in accordance with~\eqref{eq:X2}.
\item
If $i_{\ell+1} < j_\ell$, then the last 
column of~$X_{I,J}$ is a column of~1's. Consequently, the determinant
$\det(\mathbf1\,X_{I,J})$ vanishes, as
desired.
\item
If $i_{\ell+1} > j_\ell$ then the last row of $X_{I,J}$ is a row of 0's.
We may therefore expand the determinant $\det(\mathbf1\,X_{I,J})$
along the last row and obtain
$$
\det(\mathbf1\,X_{I,J})=(-1)^{\ell}
\det(X_{I\backslash\{i_{\ell+1}\},J}).
$$
If $i_1\le j_1\le i_2\le \dots\le j_\ell$, then, by
Lemma~\ref{lem:X1}, we infer
$$
\det(\mathbf1\,X_{I,J})=(-1)^{\ell}
\prod _{k=1} ^{\ell}X_{i_k,i_k}^{\chi(i_k=j_k)}
(1-X_{i_{k},i_{k}})^{\chi(j_{k-1}=i_{k})},
$$
which is in accordance with \eqref{eq:X2} since $j_\ell<i_{\ell+1}$.
If the above inequality chain does not hold then, again by
Lemma~\ref{lem:X1}, the determinant
$\det(X_{I\backslash\{i_{\ell+1}\},J})$ vanishes, and hence also
$\det(\mathbf1\,X_{I,J})$.
\end{itemize}
This completes the proof of the lemma.
\end{proof}

We are now in the position to state and prove our asymmetric extension
of Byun's minor summation formula.

\begin{theorem} \label{thm:AB}
Let $m$ and $n$ be positive integers, and
let $A$ and $B$ be $m\times n$ matrices.
Then we have
\begin{multline} \label{eq:AB}
\det\bigl(AU_nB^t+BU_nA^t+AB^t\bigr)\\
=\left(\sum_{1\le i_1\le j_1<i_2\le j_2<\cdots\le n}
\det(A^{i_1}B^{j_1}A^{i_2}B^{j_2}\cdots)\right)
\left(\sum_{1\le i_1< j_1\le i_2< j_2\le\cdots\le n}
\det(B^{i_1}A^{j_1}B^{i_2}A^{j_2}\cdots)\right),
\end{multline}
where $A^{i}$ is short for $A^{\{i\}}$, that is, for the $i$-th column
of~$A$, and where $A^{i_1}B^{j_1}A^{i_2}B^{j_2}\cdots$ denotes 
concatenation of $m$~columns, so that the last column in the
concatenation is $B^{j_{m/2}}$ or $A^{i_{(m+1)/2}}$, depending on the
parity of~$m$. Analogous conventions apply to the second factor
on the right-hand side of~\eqref{eq:AB}.
\end{theorem}

\begin{remark}
By setting $A=B$ in the above theorem, we retrieve Byun's minor
summation formula in Theorem~\ref{thm:Byun} since equalities of
indices in the sums on the right-hand side of~\eqref{eq:AB} produce
vanishing determinants.
\end{remark}

\begin{proof}[Proof of Theorem~\ref{thm:AB}]
We choose $X=U_n+\Id_n$ in Theorem~\ref{thm:main1}.
Then the left-hand sides of~\eqref{eq:1} and~\eqref{eq:1o} become
$$
\det\bigl(A(U_n+\Id_n)B^t+BU_nA^t\bigr),
$$
which is exactly the left-hand side of~\eqref{eq:AB}.

Let first $m$ be even.
By definition, we have
$$
f_{A,B}(U_n+\Id_n)=\underset {|I|=|J|=m/2}{\sum_{I,J\subseteq[n]}}
\det \bigl((U_n+\Id_n)_{I,J}\bigr) \det(A^I\,B^J).
$$
If we now use Lemma~\ref{lem:X1} with $X=U_n+\Id_n$, then we
see that $\det (U_n+\Id_n)_{I,J}=1$ if
$1\le i_1\le j_1<i_2\le j_2<\cdots\le j_{m/2}\le n$ and otherwise
the determinant equals zero.
Consequently, we obtain
\begin{align*}
f_{A,B}(U_n+\Id_n)
&={\sum_{1\le i_1\le j_1<i_2\le j_2<\cdots\le j_{m/2}\le n}}
\det(A^{\{i_1,\dots,i_{m/2}\}}B^{\{j_1,\dots,j_{m/2}\}})\\
&=(-1)^{\binom {m/2}2}
{\sum_{1\le i_1\le j_1<i_2\le j_2<\cdots\le j_{m/2}\le n}}
\det(A^{i_1}B^{j_1}A^{i_2}B^{j_2}\cdots B^{j_{m/2}}).
\end{align*}
Similarly, if we use Lemma~\ref{lem:X1} with $X=U_n$, then we
see that $\det (U_n)_{I,J}=1$ if
$1\le i_1< j_1\le i_2< j_2\le\cdots< j_{m/2}\le n$ and  otherwise
the determinant equals zero.
Therefore, we obtain
\begin{align*}
f_{B,A}\bigl(\J -(U_n+\Id_n)^t\bigr)&=f_{B,A}(U_n)\\
&={\sum_{1\le i_1< j_1\le i_2< j_2\le\cdots< j_{m/2}\le n}}
\det(B^{\{i_1,\dots,i_{m/2}\}}A^{\{j_1,\dots,j_{m/2}\}})\\
&=(-1)^{\binom {m/2}2}
{\sum_{1\le i_1< j_1\le i_2< j_2\le\cdots< j_{m/2}\le n}}
\det(B^{i_1}A^{j_1}B^{i_2}A^{j_2}\cdots A^{j_{m/2}}).
\end{align*}
If these findings are used in~\eqref{eq:1} then the result
is~\eqref{eq:AB} for the case of even~$m$.

\medskip
If $m$ is odd, we proceed completely analogously. Now the argument is
based on~\eqref{eq:1o} and Lemma~\ref{lem:X2}.
\end{proof}

In view of Theorem~\ref{thm:main2}, we have implicitly proved 
asymmetric extensions of Okada's minor summation formula in
Theorem~\ref{thm:Okada}. 

\begin{theorem} \label{thm:AB2}
Let $m$ and $n$ be positive integers, where $m$ is even.
Furthermore, let $A$ and $B$ be $m\times n$ matrices.
Then, with notation from Theorem~\ref{thm:AB}, we have
\begin{equation} \label{eq:AB22}
\sum_{1\le i_1< j_1\le i_2< j_2\le\cdots\le i_{m/2}<j_{m/2}\le n}
\det(A^{i_1}B^{j_1}A^{i_2}B^{j_2}\cdots A^{i_{m/2}}B^{j_{m/2}})
=\Pf\bigl(AU_nB^t-BU_n^tA^t\bigr)
\end{equation}
and
\begin{multline} \label{eq:AB23}
\sum_{1\le i_1\le j_1< i_2\le j_2<\cdots<i_{m/2}\le j_{m/2}\le n}
\det(A^{i_1}B^{j_1}A^{i_2}B^{j_2}\cdots A^{i_{m/2}}B^{j_{m/2}})
\\
=\Pf\bigl(A(U_n+\Id_n)B^t-B(U_n^t+\Id_n)A^t\bigr).
\end{multline}
\end{theorem}

\subsection{A Cauchy-type identity for skew Schur functions}
\label{sec:Cauchy}

Here we use Theorem~\ref{thm:AB2} from the previous subsection to derive 
a Cauchy-type identity for skew Schur functions. 

Before we can state
the identity, we need to recall a few basics from the theory of
symmetric functions; see \cite{MacdAC} and \cite[ch.~7]{StanBI}
for in-depth introductions into that theory.
Let $\mathbf x=(x_1,x_2,\dots)$ and $\mathbf y=(y_1,y_2,\dots)$
be infinite sequences of variables.
The $n$-th {\it complete homogeneous symmetric function}
$h_n(\mathbf x)=h_n(x_1,x_2,\dots)$ is defined by
$$
h_n(\mathbf x)=\sum_{1\le i_1\le i_2\le\dots\le i_n}
x_{i_1}x_{i_2}\cdots x_{i_n}, \quad \text{for }n\ge1,
$$
$h_0(\mathbf x)=1$ and $h_n(\mathbf x)=0$ for $n<0$.
A {\it partition} $\la$ is
a non-increasing sequence $(\la_1,\la_2,\dots,\la_\ell)$ of
non-negative integers, for some~$\ell$. 
Given two partitions~$\la$ and~$\mu$ with
$\mu_i\le\la_i$ for all~$i$, 
the {\it skew Schur function} $s_{\la/\mu}(\mathbf x)$
is defined by
\begin{equation} \label{eq:Schur} 
s_{\la/\mu}(\mathbf x)=\det\bigl(h_{\la_j-j-\mu_i+i}(\mathbf 
x)\bigr)_{1\le i,j\le \ell},
\end{equation}
and otherwise $s_{\la/\mu}(\mathbf x)$ is defined as zero.

\begin{theorem} \label{thm:Cauchy}
Let $m$ and $n$ be positive integers, where $m$ is even.
Then we have
\begin{multline} \label{eq:Cauchy}
\underset{|R|=|S|=m/2}{\sum_{R\cup S=[m]}}
(-1)^{\sum_{r\in R}(r-1)}
\underset{\la(J)/\la(I)\text{ \em hor.\ strip}}
{\sum_{I,J\subseteq[n],\,|I|=|J|=m/2}}
s_{\la(I)/\la(R)}(\mathbf x)\cdot s_{\la(J)/\la(S)}(\mathbf y)\\
=\Pf\left(\sum_{1\le k\le l\le n}\bigl(
h_{k-i}(\mathbf x)\cdot h_{l-j}(\mathbf y)
-h_{l-i}(\mathbf y)\cdot h_{k-j}(\mathbf x)
\bigr)\right)_{1\le i,j\le m},
%=\Pf\bigl(AU_nB^t-BU_n^tA^t\bigr).
\end{multline}
where $\la\bigl((i_1,i_2,\dots,i_{m/2})\bigr)
:=(i_{m/2}-\frac {m} {2},\dots,i_2-2,i_1-1)$, and
$\la/\mu$ is a horizontal strip if $\la_1\ge\mu_1\ge\la_2\ge\mu_2
\ge\cdots$.
\end{theorem} 

\begin{proof}
We choose $A=(h_{j-i}(\mathbf x))_{1\le i\le m,\,1\le j\le n}$
and $B=(h_{j-i}(\mathbf y))_{1\le i\le m,\,1\le j\le n}$ in
Theorem~\ref{thm:AB2}. We rewrite~\eqref{eq:AB23} in the form
\begin{align*}
\Pf\bigl(A(U_n+\Id_n)B^t-B(U_n^t+\Id_n)A^t\bigr)&=
\sum_{1\le i_1\le j_1<i_2\le j_2<\cdots\le j_{m/2}\le n}
\det(A^{i_1}B^{j_1}A^{i_2}B^{j_2}\cdots B^{j_{m/2}})\\
&=\sum_{1\le i_1\le j_1<i_2\le j_2<\cdots\le j_{m/2}\le n}
(-1)^{\binom {m/2}2}\det(A^I\,B^J).
\end{align*}
Now we do a (generalised)
Laplace expansion of the determinant on the right-hand
side with respect to the columns indexed by~$I$. In view of~\eqref{eq:Schur},
this leads  directly to~\eqref{eq:Cauchy}.
\end{proof}

\section*{AI declaration}
GPT-6 Astra and Claude Opus~5.5 were used for proofreading.
Otherwise no AI has been used. 

\bibliographystyle{plain}
\bibliography{minorsum}

\section*{Appendix: The Pfaffian is an irreducible polynomial}

\setcounter{equation}{0}%
%\global\def\thetheorem{\mbox{A.\arabic{theorem}}}
\global\def\thetheorem{\mbox{A}}
\global\def\theequation{\mbox{A.\arabic{equation}}}

\begin{lemma} \label{lem:3}
The Pfaffian of an even size skew-symmetric matrix~$Y$ is an irreducible
polynomial in the entries of\/~$Y$.
\end{lemma}

\begin{proof}
The proof is modelled after \cite[proof of Theorem~1,
pp.~176/177]{BochAA}.

The assertion is obvious if $Y$ has size~2. We may therefore assume
that $Y$ has size at least~4.
Suppose that
$$
\Pf(Y)=g(Y)\cdot h(Y),
$$
where $g(Y)$ and $h(Y)$ are polynomials in the entries of~$Y$.
The polynomial $\Pf(Y)$ is linear in each of the entries~$Y_{i,j}$
of~$Y$. In particular, it is linear in $Y_{1,2}$, and therefore
one of $g(Y)$ and $h(Y)$ contains $Y_{1,2}$, but not the
other. Without loss of generality let $g(Y)$ contain~$Y_{1,2}$.
Then $h(Y)$ must not contain any of the entries $Y_{1,j}$
and~$Y_{2,j}$ with $j\ge3$, due to the defining expansion of the
Pfaffian in terms of perfect matchings. If $h(Y)$ is not constant,
then it contains some $Y_{i,j}$ with $i,j\notin\{1,2\}$. In its turn,
this implies that $g(Y)$ must not contain~$Y_{1,i}$. However, we
already inferred that $h(Y)$ does not contain~$Y_{1,i}$. We have
reached a contradiction. Consequently, the polynomial $h(Y)$ is
actually constant, and hence $\Pf(Y)$ is irreducible.
\end{proof}

\end{document}